\documentclass{amsart}

\usepackage{amsmath, amssymb, amsthm, amscd}

\theoremstyle{plain}
\newtheorem{theorem}{Theorem}[section]

\newtheorem{lemma}[theorem]{Lemma}

\theoremstyle{definition}
\newtheorem{definition}[theorem]{Definition}
\newtheorem{example}[theorem]{Example}

\theoremstyle{remark}
\newtheorem{remark}[theorem]{Remark}

\newcommand{\A}{(\mathcal{A})}
\newcommand{\arcs}{\stackrel{\leftrightarrow}{a}\!} 
\newcommand{\F}{\mathcal{F}}

\numberwithin{equation}{section}

\begin{document}

\title[A condition for hamiltonicity in balanced bipartite digraphs]{A degree sum condition for hamiltonicity in balanced bipartite digraphs}
\author{Janusz Adamus}
\address{J.Adamus, Department of Mathematics, The University of Western Ontario, London, Ontario N6A 5B7 Canada}
\email{jadamus@uwo.ca}
\thanks{The research was partially supported by Natural Sciences and Engineering Research Council of Canada.}

\begin{abstract}
We prove that a strongly connected balanced bipartite digraph $D$ of order $2a$ is hamiltonian, provided $a\geq3$ and $d(x)+d(y)\geq 3a$ for every pair of vertices $x,y$ with a common in-neighbour or a common out-neighbour in $D$.
\end{abstract}
\maketitle


\section{Introduction}
\label{sec:intro}

In \cite{BGL}, Bang-Jensen et al. conjectured the following strengthening of a classical Meyniel theorem: If $D$ is a strongly connected digraph on $n$ vertices in which $d(u)+d(v)\geq2n-1$ for every pair of non-adjacent vertices $u,v$ with a common out-neighbour or a common in-neighbour, then $D$ is hamiltonian. 
(An \emph{in-neighbour} (resp. \emph{out-neighbour}) of a vertex $u$ is any vertex $v$ such that $vu\in A(D)$ (resp. $uv\in A(D)$).)

The conjecture has been partially verified under additional assumptions in \cite{BGY}, but has remained in its full generality a difficult open problem. The goal of the present note is to prove a bipartite analogue of the conjecture (Theorem~\ref{thm:main} below).

We consider digraphs in the sense of \cite{BG}, and use standard graph theoretical terminology and notation (see Section~\ref{sec:not} for details).

\begin{definition}
\label{def:3a}
Consider a balanced bipartite digraph $D$ with partite sets of cardinalities $a$. We will say that $D$ satisfies \emph{condition $\A$} when
\[
d(x)+d(y)\geq 3a
\]
for every pair of vertices $x,y$ with a common in-neighbour or a common out-neighbour.
\end{definition}

\begin{theorem}
\label{thm:main}
Let $D$ be a strongly connected balanced bipartite digraph with partite sets of cardinalities $a$, where $a\geq3$.
If $D$ satisfies condition $\A$, then $D$ is hamiltonian.\\
Moreover, the only non-hamiltonian strongly connected balanced bipartite digraph on $4$ vertices which satisfies condition $\A$ is the one obtained from the complete bipartite digraph $\stackrel{\leftrightarrow}{K}_{2,2}$ by removing one $2$-cycle.
\end{theorem}

\begin{remark}
\label{rem:sharp}
Although in light of the above mentioned conjecture one might expect something of order $2a$, it is worth noting that the bound of $3a$ in Theorem~\ref{thm:main} is sharp. Indeed, this follows from Example~\ref{ex:AM} below (due to Amar and Manoussakis \cite{AM}).
\end{remark}

\begin{example}
\label{ex:AM}
For $a\geq 3$ and $1\leq l<a/2$, let $D(a,l)$ be a bipartite digraph with partite sets $V_1$ and $V_2$ such that $V_1$ (resp. $V_2$) is a disjoint union of sets $R,S$ (resp. $U,W$) with $|R|=|U|=l$, $|S|=|W|=a-l$, and $A(D(a,l))$ consists of the following arcs:\\
(a) $ry$ and $yr$, for all $r\in R$ and $y\in V_2$,\\
(b) $ux$ and $xu$, for all $u\in U$ and $x\in V_1$, and\\
(c) $sw$, for all $s\in S$ and $w\in W$.\\
Then $d(r)=d(u)=2a$ for all $r\in R$ and $u\in U$, and $d(s)=d(w)=a+l$ for all $s\in S$ and $w\in W$. In particular, for odd $a$, in $D(a,(a-1)/2)$ we have $d(x)+d(y)\geq3a-1$ for every pair of non-adjacent vertices $x,y$.
Notice that $D(a,l)$ is strongly connected, but not hamiltonian.
\end{example}

A weaker version of Theorem~\ref{thm:main} was recently proved in \cite{AAY}. There, it is assumed that the inequality $d(x)+d(y)\geq3a$ is satisfied by \emph{every} pair of non-adjacent vertices $x$ and $y$. It is thus a bipartite analogue of the original Meyniel's hamiltonicity criterion for ordinary digraphs. The author is happy to acknowledge the influence of \cite{AAY} on the present work. In fact, Lemma~\ref{lem:1} and the first part of the proof of Theorem~\ref{thm:main} are direct adaptations of the ideas from \cite{AAY}, developed together with Lech Adamus and Anders Yeo.

\medskip

\section{Notation and terminology}
\label{sec:not}

A \emph{digraph} $D$ is a pair $(V(D),A(D))$, where $V(D)$ is a finite set (of \emph{vertices}) and $A(D)$ is a set of ordered pairs of distinct elements of $V(D)$, called \emph{arcs} (i.e., $D$ has no loops or multiple arcs). The number of vertices $|V(D)|$ is the \emph{order} of $D$ (also denoted by $|D|$). For vertices $u$ and $v$ from $V(D)$, we write $uv\in A(D)$ to say that $A(D)$ contains the ordered pair $(u,v)$.

For a vertex set $S \subset V(D)$, we denote by $N^+(S)$ the set of vertices in $V(D)$ \emph{dominated} by the vertices of $S$; i.e.,
\[
N^+(S)=\{u\in V(D): vu\in A(D)\text{\ for\ some\ }v\in S\}\,.
\]
Similarly, $N^-(S)$ denotes the set of vertices of $V(D)$ \emph{dominating} vertices of $S$; i.e,
\[
N^-(S)=\{u\in V(D): uv\in A(D)\text{\ for\ some\ }v\in S\}\,.
\]
If $S=\{v\}$ is a single vertex, the cardinality of $N^+(\{v\})$ (resp. $N^-(\{v\})$), denoted by $d^+(v)$ (resp. $d^-(v)$) is called the
\emph{outdegree} (resp. \emph{indegree}) of $v$ in $D$. The \emph{degree} of $v$ is $d(v)=d^+(v)+d^-(v)$.

For vertex sets $S, T \subset V(D)$, we denote by $A[S,T]$ the set of all arcs of $A(D)$ from a vertex in $S$ to a vertex in $T$. Let $\arcs(S,T) = |A[S,T]|+|A[T,S]|.$ Note that $\arcs(\{v\},V(D)\setminus \{v\}) = d(v)$. A set of vertices $\{v_1,\ldots,v_k\}$ such that $\arcs(\{v_i\},\{v_j\})=0$, for all $i\neq j$, is called \emph{independent}.

A directed cycle (resp. directed path) on vertices $v_1,\dots,v_m$ in $D$ is denoted by $[v_1,\ldots,v_m]$ (resp. $(v_1,\ldots,v_m)$). We will refer to them as simply \emph{cycles} and \emph{paths} (skipping the term ``directed''), since their non-directed counterparts are not considered in this article at all.

A cycle passing through all the vertices of $D$ is called \emph{hamiltonian}. A digraph containing a hamiltonian cycle is called a \emph{hamiltonian digraph}. A \emph{cycle factor} in $D$ is a collection of vertex-disjoint cycles $C_1,\dots,C_l$ such that $V(C_1)\cup\dots\cup V(C_l)=V(D)$.

A digraph $D$ is \emph{strongly connected} when, for every pair of vertices $u,v\in V(D)$, $D$ contains a path originating in $u$ and terminating in $v$ and a path originating in $v$ and terminating in $u$.

A digraph $D$ is \emph{bipartite} when $V(D)$ is a disjoint union of independent sets $V_1$ and $V_2$ (the \emph{partite sets}).
It is called \emph{balanced} if $|V_1|=|V_2|$. One says that a bipartite digraph $D$ is \emph{complete} when $d(x)=2|V_2|$ for all $x\in V_1$.

A \emph{matching from $V_1$ to $V_2$} is an independent set of arcs with origin in $V_1$ and terminus in $V_2$ ($u_1u_2$ and $v_1v_2$ are \emph{independent arcs} when $u_1\neq v_1$ and $u_2\neq v_2$). If $D$ is balanced, one says that such a matching is \emph{perfect} if it consists of precisely $|V_1|$ arcs.

\medskip

\section{Lemmas}
\label{sec:lemmas}

The proof of Theorem~\ref{thm:main} will be based on the following three lemmas.

\begin{lemma}
\label{lem:1}
Let $D$ be a strongly connected balanced bipartite digraph with partite sets of cardinalities $a\geq2$.
If $D$ satisfies condition $\A$, then $D$ contains a cycle factor.
\end{lemma}

\begin{proof}
Let $V_1$ and $V_2$ denote the two partite sets of $D$. Observe that $D$ contains a cycle factor if and only if there exist both a perfect matching from $V_1$ to $V_2$ and a perfect matching from $V_2$ to $V_1$. Therefore, by the K{\"o}nig-Hall theorem (see, e.g., \cite{B}), it suffices to show that $|N^+(S)|\geq|S|$ for every $S\subset V_1$ and $|N^+(T)|\geq|T|$ for every $T\subset V_2$.

For a proof by contradiction, suppose that a non-empty set $S\subset V_1$ is such that $|N^+(S)|<|S|$. Then $V_2\setminus N^+(S)\neq\varnothing$ and, for every $y\in V_2\setminus N^+(S)$, we have $d^-(y)\leq a-|S|$. Hence
\begin{equation}
\label{eq:2a-S}
d(y)\leq 2a-|S| \quad\mathrm{for\ every\ } y\in V_2\setminus N^+(S). 
\end{equation}
If $|S|=1$ then $|N^+(S)|=0$, and so the only vertex of $S$ has out-degree zero, which is impossible in a strongly connected $D$. If, in turn, $|S|=a$, then every vertex from $V_2\setminus N^+(S)$ has in-degree zero, which again contradicts strong connectedness of $D$. Therefore, $2\leq|S|\leq a-1$.
We now consider the following two cases.

\subsubsection*{Case 1.} $\frac{a}{2}<|S|\leq a-1$.\\
Since $D$ is strongly connected, we have $d^-(y)\geq 1$ for every $y\in V_2\setminus N^+(S)$. Note that $|V_2\setminus N^+(S)|\geq|V_1\setminus S|+1\geq2$. On the other hand, the vertices of $V_2\setminus N^+(S)$ are dominated only by those of $V_1\setminus S$. It follows that $V_2\setminus N^+(S)$ contains at least one pair of vertices, say $y_1$ and $y_2$, with a common in-neighbour. Condition $\A$ together with \eqref{eq:2a-S} thus imply that
\[
3a\leq d(y_1)+d(y_2)\leq 2(2a-|S|)=4a-2|S|<4a-a;
\]
a contradiction.

\subsubsection*{Case 2.} $2\leq|S|\leq\frac{a}{2}$.\\
Since $D$ is strongly connected, we have $d^+(x)\geq1$ for every $x\in S$. On the other hand, $|N^+(S)|\leq|S|-1$. It follows that $S$ contains at least one pair of vertices, say $x_1$ and $x_2$, with a common out-neighbour. Condition $\A$ thus implies that
\begin{multline}
\notag
3a\leq d(x_1)+d(x_2)=d^-(x_1)+d^+(x_1)+d^-(x_2)+d^+(x_2)\leq\\
a+(|S|-1)+a+(|S|-1)\leq 3a-2\,;
\end{multline}
a contradiction.

This completes the proof of existence of a perfect matching from $V_1$ to $V_2$. The proof for a matching in the opposite direction is analogous.
\end{proof}

\begin{lemma}
\label{lem:2}
Let $D$ be a strongly connected balanced bipartite digraph with partite sets of cardinalities $a\geq2$, which satisfies condition $\A$. Suppose that $D$ is non-hamiltonian. Then, for every $u\in V(D)$, there exists $v\in V(D)\setminus\{u\}$ such that $u$ and $v$ have a common in-neigbour or out-neighbour in $D$. 
\end{lemma}

\begin{proof}
For a proof by contradiction, suppose that $x'\in V(D)$ has no common in-neighbour or out-neighbour with any other vertex in $D$. By Lemma~\ref{lem:1}, $D$ has a cycle factor, say, $\F=\{C_1,\dots,C_l\}$, with $l\geq2$ (as $D$ is non-hamiltonian). Without loss of generality, we may assume that $x'\in V_1\cap V(C_1)$. 

Let $x'^+$ denote the successor of $x'$ on $C_1$. We have $d^-(x'^+)=1$. Indeed, for if $d^-(x'^+)\geq2$ then $x'^+$ would be a common out-neighbour of $x'$ and some other vertex from $V_1$. It follows that
\begin{equation}
\label{eq:x'+}
d(x'^+)=d^+(x'^+)+d^-(x'^+)\leq a+1.
\end{equation}
We claim that $x'^+$ has no common in-neighbour or out-neighbour with any other vertex in $V_2$. Suppose otherwise, and let $y'\in V_2$ be a vertex which shares an in-neighbour or an out-neighbour with $x'^+$. Then, by condition $\A$ and \eqref{eq:x'+}, we have
\[
3a\leq d(y')+d(x'^+)\leq d(y')+a+1,
\]
hence $d(y')\geq 2a-1$. It follows that $xy'\in A(D)$ for all $x\in V_1$ or else $y'x\in A(D)$ for all $x\in V_1$.
In the first case, $y'$ is a common out-neighbour of $x'$ and every other vertex in $V_1$, and in the second case $y'$ is a common in-neighbour of $x'$ and every other vertex in $V_1$. This contradicts the choice of $x'$. Consequently, there is no such $y'$, that is, $x'^+$ has no common in-neighbour or out-neighbour with any vertex in $V(D)$.

By repeating the above argument, one can now show that $x'^{++}$, the successor of $x'^+$ on $C_1$ has no common in-neighbour or out-neighbour with any vertex in $V(D)$, and, inductively, that no vertex of $C_1$ has a common in-neighbour or out-neighbour with any other vertex.
In particular, this means that there are no arcs in or out of $C_1$, which is not possible in a strongly connected non-hamiltonian digraph. This contradiction completes the proof of the lemma.
\end{proof}

\begin{lemma}
\label{lem:3}
Let $D$ be a strongly connected balanced bipartite digraph with partite sets of cardinalities $a\geq2$, which satisfies condition $\A$.
If $D$ is non-hamiltonian, then $d(u)\geq a$ for all $u\in V(D)$. 
\end{lemma}

\begin{proof}
This follows immediately from Lemma~\ref{lem:2}, condition $\A$, and the fact that the degree of every vertex in $D$ is bounded above by $2a$.
\end{proof}

\medskip

\section{Proof of the main result}
\label{sec:main-proof}

\subsubsection*{Proof of Theorem~\ref{thm:main}}
Let $D$ be a balanced bipartite digraph on $2a$ vertices, and let $V_1$ and $V_2$ denote its partite sets.
Suppose first that $a=2$. By Lemma~\ref{lem:1}, $D$ contains a cycle factor. If $D$ is not hamiltonian, this factor must consist of two $2$-cycles, say $C_1$ with vertices $x_1\in V_1$ and $y_1\in V_2$, and $C_2$ with vertices $x_2\in V_1$ and $y_2\in V_2$.
By strong connectedness of $D$ there must also exist at least one arc from $C_1$ to $C_2$ and one arc from $C_2$ to $C_1$. The only configuration in which $D$ is not hamiltonian is when there is precisely one such arc in each direction and they both join the same pair of vertices, say $x_1$ with $y_2$. $D$ is thus obtained from $\stackrel{\leftrightarrow}{K}_{2,2}$ by removing the $2$-cycle $[x_2,y_1]$.
\medskip

From now on, we assume that $a\geq3$.
By Lemma \ref{lem:1}, $D$ contains a cycle factor $\F=\{C_1, C_2,\dots, C_l\}$. Assume that $l$ is minimum possible, and for a proof by contradiction suppose that $l\geq 2$. Recall that $|C_i|$ denotes the order of cycle $C_i$.
Without loss of generality, assume that $|C_1|\leq|C_2|\leq\dots\leq|C_l|$.

\subsection*{Claim 1:} $\arcs(V(C_i), V(C_j))\leq \frac{|C_i|\cdot |C_j|}{2}$, for all $i\neq j$.

\subsubsection*{Proof of Claim 1.} Let $q\in \{1,2\}$, $u_i\in V(C_i)\cap V_q$ and $u_j\in V(C_j)\cap V_q$ be arbitrary. Let $u_i^+$ be the successor of $u_i$ in $C_i$ and let $u_j^+$ be the successor of $u_j$ in $C_j$. Let $\mathcal{Z}_q(u_i,u_j)$ be defined as $A(D)\cap \{u_iu_j^+,\ u_ju_i^+\}.$ If $|\mathcal{Z}_q(u_i,u_j)|=2$ for some $u_i,u_j$, then the cycles $C_i$ and $C_j$ can be merged into one cycle by deleting the arcs $u_iu_i^+$ and $u_ju_j^+$ and adding the arcs $u_iu_j^+$ and $u_ju_i^+$. This would contradict the minimality of $l$, so we may assume that
\begin{equation}
\label{eq:Z}
|\mathcal{Z}_q(u_i,u_j)|\leq 1\quad\mathrm{for\ all}\ u_i \in V(C_i)\cap V_q\ \mathrm{and}\ u_j \in V(C_j)\cap V_q\,.
\end{equation}
Now, consider an arc $uv\in A[V(C_i), V(C_j)]$ and assume $u\in V_q$. Let $v^-$ denote the predecessor of $v$ in $C_j$. Then $uv\in \mathcal{Z}_{q}(u,v^-)$. Similarly, if $uv\in A[V(C_j), V(C_i)]$, $u\in V_q$, and $v^-$ is the predecessor of $v$ in $C_i$, then $uv\in \mathcal{Z}_{q}(v^-,u)$. Therefore
\[
\arcs(V(C_i), V(C_j))\leq \sum_{q=1}^{2} \sum_{u_i\in V(C_i)\cap V_q} \sum_{u_j\in V(C_j)\cap V_q} |\mathcal{Z}_q(u_i,u_j)|\,,
\]
and hence, by \eqref{eq:Z},
\[
\arcs(V(C_i), V(C_j)) \leq
2 \cdot \frac{|C_i|}{2} \cdot \frac{|C_j|}{2}\,,
\]
which completes the proof of Claim 1.
\medskip

We now return to the proof of Theorem \ref{thm:main}. Repeatedly using Claim 1, we note that the following holds
\begin{multline}
\label{eq:V1andV2}
\arcs(V(C_1)\cap V_1, V(D)\setminus V(C_1)) + \arcs(V(C_1)\cap V_2, V(D)\setminus V(C_1))\\
= \arcs(V(C_1), V(D)\setminus V(C_1))=\sum_{j=2}^{l} \arcs(V(C_1), V(C_j))\leq \frac{|C_1|(2a-|C_1|)}{2}\,.
\end{multline}
Without loss of generality, we may assume that 
\begin{equation}
\label{eq:V1}
\arcs(V(C_1)\cap V_1, V(D)\setminus V(C_1))\leq \frac{|C_1|(2a-|C_1|)}{4},
\end{equation}
as otherwise 
\begin{equation}
\label{eq:V2}
\arcs(V(C_1)\cap V_2, V(D)\setminus V(C_1))\leq \frac{|C_1|(2a-|C_1|)}{4}.
\end{equation}
In other words, the average number of arcs between a vertex in $V(C_1)\cap V_1$ and $V(D) \setminus V(C_1)$ is bounded above by $(2a-|C_1|)/2$ (as $|V(C_1)\cap V_1| = |C_1|/2$). We now consider the following two cases.
\smallskip

\subsection*{Case 1.} $|C_1|\geq 4$.\\
Let $x_1,x_2\in V(C_1)\cap V_1$ be distinct and chosen so that $\arcs(\{x_1,x_2\}, V(D)\setminus V(C_1))$ is minimum. By the above formula we note that $\arcs(\{x_1,x_2\}, V(D)\setminus V(C_1))\leq 2a-|C_1|$.  Since any vertex in $C_1$ has at most $|C_1|$ arcs to other vertices in $C_1$ (as there are $|C_1|/2$ vertices from $V_2$ in $C_1$) and $|C_1|\leq a,$ we get that
\begin{equation}
\label{eq:3a}
d(x_1)+d(x_2)\leq 2|C_1| + 2a - |C_1|= 2a + |C_1| \leq 3a.
\end{equation}
We shall now prove that every two vertices in $V_2\cap V(C_1)$ share a common in-neighbour and that the inequality \eqref{eq:V2} holds.
To that end, we need to consider two sub-cases depending on the properties of $x_1$ and $x_2$.

Suppose first that $x_1$ and $x_2$ have a common in-neighbour or out-neighbour.
Condition $\A$ then implies that we have equality in \eqref{eq:3a}. It follows that there must be equalities in all the estimates that led to \eqref{eq:3a} as well. In particular,
\begin{gather}
\label{eq:3a-1}
\arcs(\{x_1,x_2\},V(D)\setminus V(C_1))=2a-|C_1|,\quad\mathrm{and}\\
\label{eq:3a-2}
\arcs(\{x_1\},V(C_1))=\arcs(\{x_2\},V(C_1))=|C_1|.
\end{gather}
By the choice of $x_1$ and $x_2$, it now follows from \eqref{eq:3a-1} that we have equality in \eqref{eq:V1}, and hence, by \eqref{eq:V1andV2}, the inequality \eqref{eq:V2} is satisfied.
Moreover, by \eqref{eq:3a-2}, every two vertices in $V_2\cap V(C_1)$ have a common in-neighbour, namely $x_1$.

Suppose then that $x_1$ and $x_2$ have no common in-neighbour or out-neighbour. In this case, we have
\begin{equation}
\label{eq:non-C1}
\begin{split}
|N^+(x_1)\cap (V(D)\setminus V(C_1))|+|N^+(x_2)\cap (V(D)\setminus V(C_1))|&\leq a-\frac{|C_1|}{2}\,,\\
|N^-(x_1)\cap (V(D)\setminus V(C_1))|+|N^-(x_2)\cap (V(D)\setminus V(C_1))|&\leq a-\frac{|C_1|}{2}\,,
\end{split}
\end{equation}
as well as
\begin{equation}
\notag
\begin{split}
|N^+(x_1)\cap V(C_1)|+|N^+(x_2)\cap V(C_1)|&\leq \frac{|C_1|}{2}\,, \mathrm{\ and}\\
|N^-(x_1)\cap V(C_1)|+|N^-(x_2)\cap V(C_1)|&\leq \frac{|C_1|}{2}\,.
\end{split}
\end{equation}
Hence, $d(x_1)+d(x_2)\leq2a$. Therefore, by Lemma~\ref{lem:3}, $d(x_1)=d(x_2)=a$ and, consequently, we have equalities in \eqref{eq:non-C1}. By the choice of $x_1$ and $x_2$, it follows that we have equality in \eqref{eq:V1}, and hence, by \eqref{eq:V1andV2}, the inequality \eqref{eq:V2} holds. Moreover, by Lemma~\ref{lem:2}, there exists $x'\in V_1\setminus\{x_1\}$ such that $x_1$ and $x'$ have a common in-neighbour or out-neighbour. Condition $\A$ then implies that $d(x')=2a$. In particular, every two vertices in $V_2\cap V(C_1)$ have a common in-neighbour, namely $x'$.

Next, let $y_1,y_2\in V(C_1)\cap V_2$ be distinct and chosen so that $\arcs(\{y_1,y_2\}, V(D)\setminus V(C_1))$ is minimum.
By \eqref{eq:V2}, we have $\arcs(\{y_1,y_2\}, V(D)\setminus V(C_1))\leq 2a-|C_1|$. Since any vertex in $C_1$ has at most $|C_1|$ arcs to other vertices in $C_1$ (as there are $|C_1|/2$ vertices from $V_2$ in $C_1$) and $|C_1|\leq a,$ we get that
\begin{equation}
\label{eq:3a-y}
d(y_1)+d(y_2)\leq 2|C_1| + 2a - |C_1|= 2a + |C_1| \leq 3a.
\end{equation}
Since $y_1$ and $y_2$ have a common in-neighbour, condition $\A$ implies that we have equality in \eqref{eq:3a-y}.
It follows that there must be equalities in all the estimates that led to \eqref{eq:3a-y} as well. That is,
\begin{gather}
\label{eq:3a-y-1}
\arcs(\{y_1,y_2\},V(D)\setminus V(C_1))=2a-|C_1|,\\
\label{eq:3a-y-2}
\arcs(\{y_1\},V(C_1))=\arcs(\{y_2\},V(C_1))=|C_1|,\\
\label{eq:3a-y-3}
|C_1|=a.
\end{gather}
By the choice of $y_1$ and $y_2$, it now follows from \eqref{eq:3a-y-1} and \eqref{eq:V2} that
\[
\arcs(\{y',y''\},V(D)\setminus V(C_1))=2a-|C_1|
\]
for any distinct $y',y''\in V_2\cap V(C_1)$.
Since any two such $y',y''$ have a common in-neighbour, we can repeat the above argument with $y'$ and $y''$ in place of $y_1$ and $y_2$ and conclude that \eqref{eq:3a-y-2} is satisfied by all vertices in $V_2\cap V(C_1)$. In other words, $D$ contains a complete bipartite digraph spanned on the vertices of $C_1$.

Next observe that, by minimality of $|C_1|$, \eqref{eq:3a-y-3} implies that $l=2$ and $|C_1|=|C_2|=a$.
Consequently, we can swap $C_1$ and $C_2$ and repeat the argument of Case 1 to get that $D$ contains also a complete bipartite digraph spanned on the vertices of $C_2$.
\smallskip

Now, we claim that\\
(i) $A[V(C_1)\cap V_1,V(C_2)]\neq\varnothing$ and $A[V(C_2),V(C_1)\cap V_2]\neq\varnothing$, or\\
(ii) $A[V(C_1)\cap V_2,V(C_2)]\neq\varnothing$ and $A[V(C_2),V(C_1)\cap V_1]\neq\varnothing$.\\
Indeed, condition $\A$ applied to pairs of vertices from $V(C_1)\cap V_1$ implies that there exists $x\in V(C_1)\cap V_1$ with $\arcs(\{x\},V(C_2))>0$. Similarly, there exists $y\in V(C_1)\cap V_2$ such that $\arcs(\{y\},V(C_2))>0$. Therefore, if neither (i) nor (ii) held, then all the arcs between $C_1$ and $C_2$ would need to go in the same direction (i.e., either $A[V(C_1),V(C_2)]=\varnothing$ or $A[V(C_2),V(C_1)]=\varnothing$). But such an arrangement is impossible in a strongly connected digraph.

Thus, without loss of generality we can assume that $D$ contains an arc from $V(C_1)\cap V_1$ to $V(C_2)$ and an arc from $V(C_2)$ to $V(C_1)\cap V_2$. Then, however, $D$ must be hamiltonian, because it contains complete bipartite digraphs on $V(C_1)$ and on $V(C_2)$. This contradiction completes the proof of Case 1.
\medskip

\subsection*{Case 2.} $|C_1|< 4$.\\
In this case $|C_1|=2$. Let $V(C_1)\cap V_1 = \{x_1\}$ and $V(C_1)\cap V_2 = \{y_1\}$. Note that, by \eqref{eq:V1}, we have $d(x_1)\leq 2 + (2a-|C_1|)/2 = a+1$. By Lemma~\ref{lem:2}, $x_1$ shares a common in-neighbour or out-neighbour with a vertex, say $x'$, in $V_1\setminus\{x_1\}$. By condition $\A$, $d(x')\geq 2a-1$, and so
\begin{equation}
\label{y1-all-neighbour}
x'y\in A(D) \mathrm{\ for\ all\ } y\in V_2\qquad\mathrm{or\ else}\qquad yx'\in A(D) \mathrm{\ for\ all\ } y\in V_2.
\end{equation}
That is, $y_1$ has a common in-neighbour with every vertex in $V_2\setminus\{y_1\}$ or else $y_1$ has a common out-neighbour with every vertex in $V_2\setminus\{y_1\}$. 
The remainder of the proof of this case is divided into two sub-cases depending on the actual value of $d(x_1)$.

\subsubsection*{Case 2a.} $d(x_1)=a+1$.\\
Then, by \eqref{eq:V1andV2}, $d(y_1)\leq a+1$. Hence, by \eqref{y1-all-neighbour} and condition $\A$, we have
\begin{equation}
\label{eq:all-y}
d(y)\geq2a-1 \mathrm{\ \ for\ all\ } y\in V_2\setminus\{y_1\}.
\end{equation}
It follows that, for every $y\in V_2\setminus\{y_1\}$, at least one of the arcs $x_1y,yx_1$ belongs to $A(D)$. Moreover, every $x\in V_1\setminus\{x_1\}$ shares a common in-neighbour or out-neighbour with $x_1$, and so
\begin{equation}
\label{eq:all-x}
d(x)\geq 2a-1 \mathrm{\ \ for\ all\ } x\in V_1\setminus\{x_1\}.
\end{equation}
We now claim that, for every $x\neq x_1$, at most one of the arcs $xy_1,y_1x$ is contained in $A(D)$. Suppose otherwise, and let $\tilde{x}\in V_1\setminus\{x_1\}$ be such that $\tilde{x}y_1,y_1\tilde{x}\in A(D)$. Say, $\tilde{x}\in V(C_j)$ for some $j\neq1$. Let $\tilde{x}^+$ (resp. $\tilde{x}^-$) denote the successor (resp. predecessor) of $\tilde{x}$ on $C_j$. By~\eqref{eq:all-y}, one of the following must hold:
\begin{itemize}
\item[(i)] $x_1\tilde{x}^+\in A(D)$, or
\item[(ii)] $\tilde{x}^-x_1\in A(D)$, or else
\item[(iii)] $x_1\tilde{x}^+\notin A(D)$, $\tilde{x}^-x_1\notin A(D)$, and $\tilde{x}^+x_1, x_1\tilde{x}^-\in A(D)$.
\end{itemize}
In the first case, one can merge $C_1$ with $C_j$ by replacing the arc $\tilde{x}\tilde{x}^+$ on $C_j$ with the path $(\tilde{x},y_1,x_1,\tilde{x}^+)$. This contradicts the minimality of $l$. In the second case, one can merge $C_1$ with $C_j$ by replacing the arc $\tilde{x}^-\tilde{x}$ on $C_j$ with the path $(\tilde{x}^-,x_1,y_1,\tilde{x})$. This contradicts the minimality of $l$. In the third case, in turn, both $\tilde{x}^+$ and $\tilde{x}^-$ are joined by symmetric arcs with every vertex in $V_1\setminus\{x_1\}$, by \eqref{eq:all-y}. One can thus merge $C_1$ with $C_j$ by replacing the path $(\tilde{x}^{--},\dots,\tilde{x}^{++})$ on $C_j$ with the path $(\tilde{x}^{--},\tilde{x}^+,x_1,y_1,\tilde{x},\tilde{x}^-,\tilde{x}^{++})$, where $\tilde{x}^{++}$ (resp. $\tilde{x}^{--}$) denotes the successor of $\tilde{x}^{+}$ (resp. predecessor of $\tilde{x}^{-}$) on $C_j$. This again contradicts the minimality of $l$, which completes the proof of our claim. (Note that the above argument works whenever $|C_j|\geq4$. If $|C_j|=2$, however, there is nothing to prove, given that $\tilde{x}y_1,y_1\tilde{x}\in A(D)$ and one of (i)-(iii) holds.)

By \eqref{eq:all-x}, we now get that every $x\neq x_1$ is joined by symmetric arcs with all vertices in $V_2\setminus\{y_1\}$. In other words, $D$ contains a complete bipartite digraph spanned by the vertices $V(D)\setminus\{x_1,y_1\}$. Moreover, by \eqref{eq:all-y} and \eqref{eq:all-x}, we have $\arcs(\{x_1\},\{y\})\geq1$ and $\arcs(\{y_1\},\{x\})\geq1$ for all $y\neq y_1, x\neq x_1$. Since in a strongly connected digraph it cannot happen that $A[V(C_1),V(D)\setminus V(C_1)]=\varnothing$ or $A[V(D)\setminus V(C_1),V(C_1)]=\varnothing$, it follows that there exist vertices $\tilde{x},\tilde{y}\in V(D)\setminus V(C_1)$ such that $x_1\tilde{y}, \tilde{x}y_1\in A(D)$ or $\tilde{y}x_1,y_1\tilde{x}\in A(D)$. One can readily see that then $D$ contains a Hamilton cycle. This contradiction completes the proof of \emph{Case 2a}.

\subsubsection*{Case 2b.} $d(x_1)=a$.\\
Since $a\geq3$, it follows that there exists $\tilde{y}\in V_2\setminus\{y\}$ such that $x_1\tilde{y}\in A(D)$ or $\tilde{y}x_1\in A(D)$.
Say, $\tilde{y}\in V(C_j)$ for some $j\neq1$.
Let $\tilde{y}^+$ (resp. $\tilde{y}^-$) denote the successor (resp. predecessor) of $\tilde{y}$ on $C_j$.
If $x_1\tilde{y}\in A(D)$, then $\tilde{y}$ is a common out-neighbour of $x_1$ and $\tilde{y}^-$, and so $d(\tilde{y}^-)=2a$, by condition $\A$. In particular, $\tilde{y}^-y_1\in A(D)$, and hence $C_1$ can be merged with $C_j$ by replacing the arc $\tilde{y}^-\tilde{y}$ on $C_j$ with the path $(\tilde{y}^-,y_1,x_1,\tilde{y})$. This contradicts the minimality of $l$.
If, in turn, $\tilde{y}x_1\in A(D)$, then $\tilde{y}$ is a common in-neighbour of $x_1$ and $\tilde{y}^+$, and so $d(\tilde{y}^+)=2a$, by condition $\A$. In particular, $y_1\tilde{y}^+\in A(D)$, and hence $C_1$ can be merged with $C_j$ by replacing the arc $\tilde{y}\tilde{y}^+$ on $C_j$ with the path $(\tilde{y},x_1,y_1,\tilde{y}^+)$. This again contradicts the minimality of $l$, which completes the proof of the theorem.
\qed

\medskip

\end{document}